\documentclass[11pt,a4paper,reqno]{amsart}

\usepackage{color}
\usepackage{amsmath,enumitem}
\usepackage{amsfonts}
\usepackage{amssymb}
\usepackage{amsbsy}
\usepackage{amscd}
\usepackage{amsthm}
\usepackage{mathrsfs}
\usepackage[english]{babel}
\usepackage{graphicx}
\usepackage[all]{xy}
\usepackage{mathtools}

\usepackage{tikz} 
\usetikzlibrary{calc,intersections,through,
  backgrounds,arrows,decorations.pathmorphing,backgrounds,positioning,
fit,petri}   
\usetikzlibrary{calc}
\usepackage{tikz-cd} 

\usepackage{ytableau}

\setlist{nosep}

\usepackage[active]{srcltx}
\usepackage[parfill]{parskip}

\newtheorem{theorem}{Theorem}[section]
\newtheorem{lemma}[theorem]{Lemma}
\newtheorem{definition}[theorem]{Definition}
\newtheorem{corollary}[theorem]{Corollary}
\newtheorem{proposition}[theorem]{Proposition}
\newtheorem{remark}[theorem]{Remark}

\newcommand{\Hom}{{\mathrm{Hom}}}

\newcommand{\id}{{\rm id}}

\newcommand{\eins}{\leavevmode\hbox{\small1\kern-3.8pt\normalsize1}}

\newcommand{\residue}{{\rm res}}
\newcommand{\content}{{\rm con}}

\newcommand{\mN}{\mathbb{N}}

\newcommand{\mK}{\mathbb{K}}
\newcommand{\mk}{\Bbbk}
\newcommand{\ZZ}{\mathbb{Z}}

\newcommand{\ft}{{\mathfrak t}}

\DeclarePairedDelimiter\abs{\lvert}{\rvert}%
\DeclarePairedDelimiter\norm{\lVert}{\rVert}%
\makeatletter
\let\oldabs\abs
\def\abs{\@ifstar{\oldabs}{\oldabs*}}
\let\oldnorm\norm
\def\norm{\@ifstar{\oldnorm}{\oldnorm*}}
\makeatother

\begin{document}

\title{The blocks of the periplectic Brauer algebra in positive characteristic}

\author{Sigiswald Barbier}
\email{Sigiswald.Barbier@UGent.be}
\address{Department of Mathematical Analysis, Faculty of Engineering
  and Architecture, Ghent University, Krijgslaan 281, 9000 Gent, Belgium}
\author{Anton Cox}
\email{A.G.Cox@city.ac.uk}
\author{Maud De Visscher}
\email{Maud.Devisscher.1@city.ac.uk}
\address{Mathematics Department, City, University of London,
  Northampton Square, London, EC1V 0HB, England}
\date{7 February 2019}

\begin{abstract}
We determine the blocks of the periplectic Brauer algebra over any
field of odd positive 
characteristic. 
\end{abstract}

\maketitle
\section{Introduction}
The periplectic Brauer algebra belongs to a class of algebras which
can be represented using diagrams. Other examples of diagram algebras
include the symmetric group algebra, the Hecke algebra, the Temperley-Lieb
algebra and the Brauer algebra.

The periplectic Brauer algebra was first introduced by Moon
\cite{Moon} to study the periplectic Lie superalgebra.  This is
similar to the way Schur-Weyl duality is used to relate representation
theory of the symmetric group to representation theory of the general
linear group and representation theory of the Brauer algebra to the
orthogonal and symplectic Lie algebra or to the encompassing
orthosymplectic Lie superalgebra \cite{BSR, EhrigStroppel,
  LehrerZhang}.

The periplectic Brauer algebra $A_n$ is closely related to the Brauer
algebra $B_n(\delta)$ for $\delta=0$, \cite{Serganova, Kujawa}. For
example, they can both be represented using the same Brauer diagrams
and with multiplication only differing up to a minus sign. So it
should come as no surprise that certain aspects of the representation
theory of the periplectic Brauer algebra resembles the representation
theory of the Brauer algebra. For instance, their simple modules can
be labelled by the same partitions.  However, there are also striking
differences. While the Brauer algebra is cellular, this is no longer
the case for the periplectic Brauer algebra. Also the description of
the blocks in characteristic zero is quite different, and we will show
this is still the case in positive characteristic.

The representation theory of the Brauer algebra has already been
developed for a long time \cite{KX, CDM, CDM2, King, Martin}. In
contrast the  representation theory of the periplectic Brauer
algebra remained unstudied until quite recently. In particular the
simple modules have been classified for arbitrary characteristic
\cite{Kujawa} and for characteristic zero (or large characteristic) a classification of the
blocks \cite{Coulembier} and a complete description of the
decomposition multiplicities \cite{CoulembierEhrig} have been
obtained. 

Calculating the decomposition multiplicities of the (periplectic)
Brauer algebra in positive characteristic is an important open
problem. Since this is related to the long-standing open problem of
the decomposition multiplicities of the symmetric group, a solution to
this problem seems currently not within reach.  Instead, as a first
step, we obtain in this paper a complete classification of the blocks
of the periplectic Brauer algebra in all (odd) positive characteristic. 

There are a number of technical challenges that have to be addressed
along the way. As the periplectic Brauer algebras are not cellular,
we need to work in the setting of standardly based algebras, where a
more limited set of tools are at our disposal. There is also a much
more delicate interplay with the representation theory of the
symmetric group: for example we develop a version of BGG reciprocity
for these algebras, but with a twist arising from the (highly
non-trivial) Mullineux map on representations of the symmetric
group (Theorem \ref{bgg}). 

This leads to a very different classification of blocks from the
classical Brauer case in Theorem~\ref{Block decomposition}. For example we will see that if $n$ is not too
small compared to the characteristic then there is only one block (and
will give a complete classification in terms of certain staircase
partitions in the general case). Note
that this result is better than in the Brauer algebra case where only
the limiting blocks are known in positive characteristic \cite{King}.

This paper is structured as follows. In Section \ref{Sec periplectic
  Brauer algebra} we introduce the periplectic Brauer algebra. As
mentioned, the periplectic Brauer algebra fails to be cellular, but it
still satisfies the weaker notion of a standardly based algebra. In
Section \ref{sba}, we recall the definitions and properties of
standardly based algebras needed in this paper. We also review the
relevant partition combinatorics and representation theory of the
symmetric group algebra in Section \ref{symstuff}.  In Section
\ref{Section standard modules} we combine all this information to
obtain a standard basis of the periplectic Brauer algebra and an
explicit description of the standard modules. Using localisation and
globalisation functors in Section \ref{Section embeddings}, we obtain a full
embedding of the module category $A_n\text{-mod}$ into
$A_{n+2}\text{-mod}$. This can then be used for induction
arguments. In particular, using the results of Sections  \ref{Section standard modules} and \ref{Section
  embeddings} we prove a  BGG-reciprocity  for the periplectic Brauer algebra in arbitrary characteristic in Section~\ref{Sec
   BGG}. Finally, we obtain a complete description of the blocks of the
periplectic Brauer algebra in any characteristic $p\neq 2$  in Section \ref{Sec blocks}.

\medskip
We will wish to compare the representation theory of our algebras in
characteristic $p$ to characteristic zero, {\it and so will use the
  following conventions throughout this paper}. We will consider a
$p$-modular system $(\mK, R, \mk)$, which means that $R$ is a discrete
valuation ring, $\mK$ is the associated field of fractions, which will
be of characteristic zero, and $\mk$ is the quotient of $R$ by its
maximal ideal, and is a field of characteristic $p>0$.  If we wish to
consider an arbitrary field we will denote it by $k$.

Throughout the paper, unless otherwise stated, all the modules will be
left modules.

\section{The periplectic Brauer algebra} \label{Sec periplectic Brauer algebra}

The periplectic Brauer algebra  was first introduced by Moon
\cite{Moon} in terms of generators and relations. However we will use
the diagrammatic definition due to Kujawa and Tharp
\cite{Kujawa}. This is very similar to the definition of the Brauer
algebra in terms of Brauer diagrams, but with a deformed version of
multiplication as defined below.

An \emph{$(r,s)$-Brauer diagram} is a partition of $r+s$ nodes into
(unordered) pairs.  It can be represented pictorially by $r$ nodes on
a (northern) horizontal line and $s$ nodes on a second (southern)
horizontal line below the first one, with an edge between two
nodes if they belong to the same pair. An edge which connects two
nodes on the northern line is called a \emph{cup}, an edge which connects two
nodes on the southern line is called a \emph{cap}, and an edge which connects
a node on the northern line with a node on the southern line is called
a \emph{propagating line}.  An example of a $(6,8)$-Brauer diagram is given
in Figure \ref{diagex}.

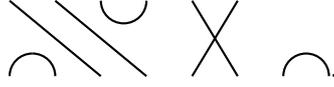
\begin{figure}[ht]
$$
\begin{tikzpicture}[scale=1,thick,>=angle 90]
\begin{scope}[xshift=4cm]
\draw (0,1) -- +(1.2,-1);
\draw (0.6,1) -- +(1.2,-1);
\draw (0,0) to [out=90, in=180] +(0.3,0.3);
\draw (0.6,0) to [out=90, in=0] +(-0.3,0.3);
\draw (1.2,1) to [out=-90, in =180] +(0.3,-0.3) to [out=0, in =-90] +(0.3,0.3);
\draw (2.4,1) -- + (0.6,-1);
\draw (3,1) -- + (-0.6,-1);
\draw (3.6,0) to [out=90, in=180] +(0.3,0.3) to [out=0, in=90] +(0.3,-0.3);
\end{scope}
\end{tikzpicture}.
$$
\caption{An example of a $(6,8)$-Brauer diagram}\label{diagex}
\end{figure}

To compose an $(r,s)$-Brauer diagram $d_1$ with an $(s,t)$-Brauer
diagram $d_2$ we will need the notion of a marked Brauer diagram. We
decorate Brauer diagrams with certain markings as follows, and will
choose a preferred decoration to be the standard marking. 
  
To produce a \emph{marked Brauer diagram} we choose a marking for each cup
and cap, and a linear order on them, as follows. Each 
cup is marked with a diamond $\Diamond$, and each cap with either a
left arrow $\lhd$ or a right arrow $\rhd$. Given an arbitrary linear
ordering on the caps and cups, we depict this by placing the
markings on different latitudes between the northern and southern
horizontal lines, such that a cup or cap which is larger than another
cup or cap has the more northerly latitute of the two. (In order to do
this we may need to deform our Brauer diagram, but as these are only
considered up to isotopy this does not affect the definition.) We say
that two markings are \emph{adjacent} if there is no 
other marking between them in this order.

We can now define a \emph{standard marking} for each Brauer diagram.  First
we choose to mark all caps with right arrows. For our ordering, 
we set all the cups to be larger than all the caps. Then we say that
one cup is larger than another cup if its leftmost node is to the left
of the leftmost node of the other cup.  Finally we say that a cap is
larger than another 
cap if its leftmost node is to the right of the leftmost node of the
other cap. 

Two examples of Brauer diagrams with markings are given in Figure
\ref{secondex}.  The righthand diagram has the standard marking.

\begin{figure}[ht]
$$
\begin{tikzpicture}[scale=1,thick,>=angle 90]
\begin{scope}[xshift=0cm]
\draw (0,1) to [out=-20, in=120]   +(1.2,-1);
\draw (0.6,1) to [out=-60, in=160]   +(1.2,-1);
\draw (0,0) to [out=90, in=180] +(0.3,0.7)to [out=0, in=90] +(0.3,-0.7);
\draw (0.3, 0.7) node[]{$\lhd$};
\draw (1.2,1) to [out=-90, in =180] +(0.3,-0.7) to [out=0, in =-90] +(0.3,0.7);
\draw (1.5, 0.3) node[]{$\Diamond$};
\draw (2.4,1) -- + (0.6,-1);
\draw (3,1) -- + (-0.6,-1);
\draw (3.6,0) to [out=90, in=180] +(0.3,0.2) to [out=0, in=90] +(0.3,-0.2);
\draw (3.9, 0.2) node[]{$\rhd$};
\end{scope}
\begin{scope}[xshift=8cm]
\draw (0,1) -- +(1.2,-1);
\draw (0.6,1) -- +(1.2,-1);
\draw (0,0) to [out=90, in=180] +(0.3,0.2);
\draw (0.6,0) to [out=90, in=0] +(-0.3,0.2);
\draw (0.3, 0.2) node[]{$\rhd$};
\draw (1.2,1) to [out=-90, in =180] +(0.3,-0.3) to [out=0, in =-90] +(0.3,0.3);
\draw (1.5, 0.7) node[]{$\Diamond$};
\draw (2.4,1) -- + (0.6,-1);
\draw (3,1) -- + (-0.6,-1);
\draw (3.6,0) to [out=90, in=180] +(0.3,0.4) to [out=0, in=90] +(0.3,-0.4);
\draw (3.9, 0.4) node[]{$\rhd$};
\end{scope}
\end{tikzpicture}
$$
\caption{A marked and a standardly marked Brauer diagram}
\label{secondex}
\end{figure}
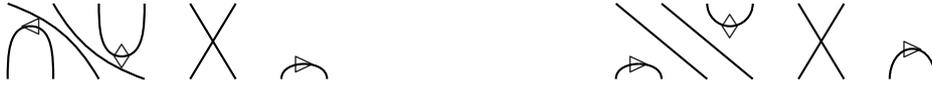

Using these markings we define the composition $d_1 d_2$ of Brauer
diagrams $d_1$ and $d_2$ as follows. If the number of nodes on the
southern line of $d_1$ is different from the number of nodes on the
northern line of $d_2$, we set $d_1d_2$ equal to zero. Otherwise we
concatenate the two Brauer diagrams by identifying the northern
horizontal line of $d_2$ with the southern line of $d_1$ to obtain a
new diagram $d_1\star d_2$. If this new diagram contains closed loops,
we set $d_1 d_2$ zero. Otherwise $d_1\star d_2$ is again a Brauer diagram
and we set
\begin{align} \label{Composition Brauer diagrams}
 d_1 d_2 = (-1)^{\gamma(d_1,d_2)} d_1 \star d_2
\end{align}
where $\gamma(d_1,d_2)$ is defined as follows.

First equip $d_1$ and $d_2$ with their standard marking. This will
give us a decoration on $d_1 \star d_2$, possibly with more than one
marking on the same edge. We can make this into a standard marking by
combinations of the following two operations: (i) if two markings are
adjacent, we switch their order, and (ii) if an arrow and a diamond
are adjacent and on the same edge, we remove both markings.  Then
$\gamma(d_1,d_2)$ counts the number of switching operations of type
(i) and the number of cancelling operations of type (ii) \emph{where the
arrow points away from the diamond} which are needed to obtain a
standard marking on $d_1 \star d_2$. Of course there may be many
different ways to obtain the standard marking, but it is shown in
\cite{Kujawa} that $(-1)^{\gamma(d_1,d_2)}$ is independent of the
chosen operations.

\begin{definition}
Let $k$ be a field. The periplectic Brauer algebra $A_n$ is the
$k$-algebra with basis given by $(n,n)$-Brauer diagrams, and
multiplication given by linear extension of the composition of
Brauer diagrams defined in \eqref{Composition Brauer diagrams}.
\end{definition}

We may consider the category whose objects are natural numbers and
where morphisms between $r$ and $s$ are given by $(r,s)$-Brauer
diagrams. This can be given the structure of a (strict) monoidal supercategory in the sense of \cite{BrundanEllis} by
defining the tensor product of an $(r,s)$-Brauer diagram with an
$(r',s')$-Brauer diagram as follows. We concatenate horizontally the first diagram with  $r'$ (non-crossing) propagating lines on the right and concatenate the second diagram with $s$ (non-crossing) propagating lines on the left. Then we take the composition defined in \eqref{Composition Brauer diagrams} of the $(r+r',s+r')$ and $(s+r', s+s')$ diagrams obtained in this way. This gives us a $(r+r',s+s')$-Brauer diagram with a possible minus sign. 
Then the monoidal supercategory is generated by the elements $I$, $X$,
$\cap$ and $\cup$ where $I$ is the unique $(1,1)$-Brauer diagram,
$\cap$ and $\cup$ are the unique $(0,2)$- and $(2,0)$-Brauer diagrams,
and $X$ is the $(2,2)$-Brauer diagram with two propagating lines that
cross precisely once \cite[Theorem 3.2.1]{Kujawa}. 

\begin{definition}\label{antiauto} We define a contravariant autoequivalence $\phi$ on this
  category as follows.  The map $\phi$ fixes objects, and we set 
\begin{align*}
\phi(I)=I, \qquad, \phi(X)=-X, \qquad \phi(\cap)=\cup,\qquad \phi(\cup) =-\cap
\end{align*}
and require $\phi(ab)= \phi(b)\phi(a)$ and $\phi(a\otimes b)=
\phi(a)\otimes\phi(b)$. 
One can check that $\phi$ is well-defined and induces an algebra
anti-automorphism $\phi$ on $A_n$, see 
\cite[Section 2.1.6]{Coulembier}.
\end{definition}

The symmetric group algebra ${\rm H}_n = k\mathfrak{S}_n$ appears as the subalgebra of $A_n$ spanned by all diagrams with no caps (or cups). It also appears as the quotient of $A_n$ by the ideal generated by all diagrams
containing at least one cap (or cup).  It is easy to see that the anti-automorphism $\phi$ induces an
anti-automorphism $\phi$ on this subalgebra or quotient given by
$$\phi(w) =(-1)^{\ell(w)} w^{-1}$$
for all $w\in \mathfrak{S}_n$.

The study of the representation theory of the Brauer algebras in
\cite{CDM, CDM2} was based on the fact that these algebras were cellular
(and frequently even quasihereditary). Unfortunately there is no
obvious cellular structure in the marked Brauer case. However
in \cite[Theorem 4.1.2]{Coulembier}, it is shown that $A_n$ is a
standardly based algebra. This will be the general framework in which
we need to work, and so we now recall the definition and some of the
properties of such classes of algebras.

\section{Standardly based algebra}\label{sba}

We recall the notion of standardly based algebras introduced in
\cite{DuRui}.  In this section we work over an arbitrary field
$k$. Let $(\Lambda, \geq)$ be a poset, $A$ be a finite dimensional
$k$-algebra and $\mathcal{B}$ be a basis for $A$. We say that $(A,
\mathcal{B})$ is a \emph{based algebra} if we can write $\mathcal{B}$
as a disjoint union of subsets $\mathcal{B}^\lambda$ for $\lambda \in
\Lambda$ such that for all $a\in A$ and $b\in \mathcal{B}^\lambda$ we
have that $ab$ and $ba$ can be written as linear combinations of basis
elements $c\in \mathcal{B}^\mu$ with $\mu \geq \lambda$. This allows
us to define two-sided ideals of $A$ for each $\lambda \in \Lambda$,
namely $A^{\geq \lambda}$ spanned by $\cup_{\mu \geq \lambda}
\mathcal{B}^\mu$ and $A^{> \lambda}$ spanned by $\cup_{\mu > \lambda}
\mathcal{B}^\mu$. We define $A^{\lambda}$ to be the $(A,A)$-bimodule
$A^{\geq \lambda} / A^{> \lambda}$. We abuse notation and view
$A^\lambda$ as the $k$-span of $\mathcal{B}^\lambda$.  If we assume
further that for each $\lambda \in \Lambda$ we have indexing sets
$I(\lambda)$ and $J(\lambda)$ such that
$$\mathcal{B}^\lambda = \{ a_{ij}^\lambda \, | \, (i,j)\in
I(\lambda)\times J(\lambda)\}$$ 
and for each $a\in A$ and  $a_{ij}^\lambda$ we have
\begin{eqnarray*}
&& a \, a_{ij}^\lambda = \sum_{i'\in I(\lambda)} f_{i', \lambda}(a,i)
  \, a_{i'j}^\lambda \, {\rm mod} A^{> \lambda} \,\, \mbox{and} \\ 
&& a_{ij}^\lambda \, a = \sum_{j'\in J(\lambda)} f_{\lambda,
    j'}(j,a)\, a_{ij'}^\lambda \, {\rm mod} A^{> \lambda}, 
\end{eqnarray*}
then we say that $(A, \mathcal{B})$ is a \emph{standardly based algebra}.

Note that if, in addition, we have an algebra anti-involution $\psi$
such that $\psi(a_{ij}^\lambda) = a_{ji}^\lambda$ then
$(A,\mathcal{B})$ is a cellular algebra as defined in \cite{GL}.

Now for each $(i_0, j_0)\in I(\lambda) \times J(\lambda)$ we can
define the left $A$-module $\Delta(\lambda , j_0)$ (resp. the right
$A$-module $\Delta(i_0,\lambda)$) to be the $k$-span of $\{
a_{ij_0}^\lambda \, | \, i\in I(\lambda)\}$ (respectively of $\{
a_{i_0j}^\lambda \, | \, j\in J(\lambda)\}$). As these modules are
clearly independent of the choice of $(i_0,j_0)$ we write
$\Delta(\lambda) = \Delta(\lambda, j_0)$ and $\Delta^{\rm op}(\lambda)
= \Delta(i_0, \lambda)$ and call them the left, respectively right,
\emph{standard modules} for $A$. By definition, we have an isomorphism
of $(A,A)$-bimodules $A^{\lambda} \cong \Delta(\lambda) \otimes_{k}
\Delta^{\rm op} (\lambda)$.

It is shown in \cite[2.4]{DuRui} that there is a subset $\Lambda'
\subseteq \Lambda$ (defined in terms of a bilinear form on standard
modules) such that for all $\lambda \in \Lambda'$ we have that
$L(\lambda) := \Delta(\lambda)/{\rm rad}\Delta(\lambda)$ is simple and
moreover $\{ L(\lambda) \, | \, \lambda \in \Lambda'\}$ forms a
complete set of pairwise non-isomorphic simple (left) $A$-modules. For
each $\lambda \in \Lambda'$ we denote by $P(\lambda)$ the projective
cover of $L(\lambda)$.

\begin{proposition}\label{standardproperties} \cite[(2.4.1),(2.4.4)]{DuRui}
Let $\lambda\in \Lambda'$ and $\mu \in \Lambda$.
\begin{enumerate}
\item The composition multiplicity $[\Delta(\mu):L(\lambda)]$ satisfies
$$[\Delta(\mu):L(\lambda)] \neq 0 \,\, \mbox{implies that} \,\,
  \lambda \leq \mu$$ 
and $[\Delta(\lambda):L(\lambda)] = 1$.
\item The projective indecomposable module $P(\lambda)$ has a
  filtration by standard modules. If we denote by
  $(P(\lambda):\Delta(\mu))$ the number of sections isomorphic to
  $\Delta(\mu)$ in this filtration then we have
$$(P(\lambda):\Delta(\mu)) = \dim (\Delta^{\rm op}(\mu)\otimes_A P(\lambda)).$$
In particular we have
$$(P(\lambda):\Delta(\mu)) \neq 0 \,\, \mbox{implies that} \,\, \mu
\geq \lambda$$ 
and $(P(\lambda):\Delta(\lambda)) = 1$.
\end{enumerate}
\end{proposition}

For $\lambda, \mu \in \Lambda'$ we say that the two simple
$A$-modules $L(\lambda)$ and $L(\mu)$ belong to the same
\emph{block} if there is a sequence $\lambda = \lambda_1, \lambda_2,
\ldots , \lambda_t = \mu$ in $\Lambda'$ and a sequence of
indecomposable $A$-modules $M_1, M_2, \ldots M_{t-1}$ such that for
each $1\leq i\leq t-1$ we have that $L(\lambda_i)$ and
$L(\lambda_{i+1})$ appear as composition factors of $M_i$. This gives
an equivalence relation on $\Lambda'$ where each equivalence
class corresponds to a block of $A$. Thus we will abuse notation and
refer to elements $\lambda$ and $\mu$ of $\Lambda'$ as being in the
same block if $L(\lambda)$ and $L(\mu)$ lie in the same block for $A$.

\begin{corollary}\label{stdinP}
Every standard module occurs in the filtration of some projective
indecomposable module, and hence has all composition factors belonging
to a single block.
\end{corollary}

\begin{proof}
Let $\Delta(\mu)$ be a standard module for $A$. By
Proposition \ref{standardproperties}(2) it is enough to show that 
\begin{equation}\label{nonzero}
\dim(\Delta^{\rm op}(\mu)\otimes_A P(\lambda))\neq 0
\end{equation}
for some $\lambda\in\Lambda'$. But $A\cong
\bigoplus_{\lambda\in\Lambda'}P(\lambda)^{a_\lambda}$ for some
$a_{\lambda}>0$ and 
$$\dim(\Delta(\mu)^{\rm op}\otimes_AA)=\dim(\Delta(\mu)^{\rm op})\neq
  0$$
and so there must exist some $\lambda$ such that (\ref{nonzero}) holds.
\end{proof}

By Corollary \ref{stdinP} it makes sense to talk about the block of a
standard module 
and hence to extend the block relation from $\Lambda'$ to the
whole of $\Lambda$.

\section{Partition combinatorics and representations of the symmetric
  group}
\label{symstuff}

\subsection*{Partition combinatorics}

We briefly review the basic properties of partitions and Young digrams
that we will need; further details can be found in \cite{JK}. Given
$n\in \mathbb{N}$, a \emph{partition} $\lambda$ of $n$ is an element
$\lambda= (\lambda_1, \lambda_2, \dots, \lambda_k)$ such that
$\lambda_i \in \mN$ for all $1\leq i\leq k$, with $\lambda_1 \geq
\lambda_2\geq \dots \geq \lambda_k$ and $|\lambda|:=\lambda_1+
\lambda_2 + \dots \lambda_k=n$. If $\lambda$ is a partition of $n$ we
write $\lambda \vdash n$.

We define a partial order $\unlhd$ on the set of all partitions (of
any $n$) as follows. For two partitions $\lambda$ and $\mu$ we set
$\lambda \unlhd \mu$ if and only if either $|\lambda|>|\mu|$ or
$|\lambda| = |\mu|$ and $\sum_{i=1}^{j} \lambda_i\leq \sum_{i=1}^j
\mu_i$ for all $j\geq 1$. (If $|\lambda|=|\mu|$ this is just the usual
dominance order.)

We will often identify a partition $\lambda$ with its \emph{Young
diagram}. This is a collection of $n$ boxes ordered in left-justified
rows such that the $i$th row from the top contains $\lambda_i$ boxes. For
example the Young diagram corresponding to $\lambda= (4,4,2,1)$ is
illlustrated in Figure \ref{young}(a).  The
\emph{transpose} $\lambda^T$ of a partition $\lambda$ is the partition whose
corresponding Young diagram has as rows the columns of the Young
diagram of $\lambda$. For example, the transpose of $\lambda=
(4,4,2,1)$ is $\lambda^T= (4,3,2,2)$.

A \emph{hook} of a Young diagram corresponding to a given box consists of
that box and all boxes below it in the same column and all boxes to
the right of it in the same row. We
obtain the corresponding \emph{rim hook} by pushing all boxes down and to the
right until they get to the edge of the Young diagram.  For example,
the hook corresponding to the box $(1,2)$ of $\lambda=(4,4,2,1)$ is
shown in Figure \ref{young}(b), and the corresponding rim hook is
shown in Figure \ref{young}(c).

\begin{figure}[ht]
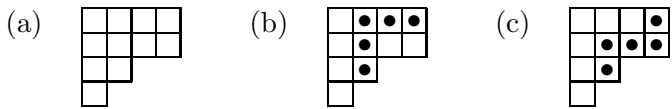

\[\text{ (a) }\quad \ytableausetup{mathmode, boxsize=0.8em} \ydiagram{4,4,2,1}
\qquad \text{ (b) }\quad
\ytableausetup{mathmode, boxsize=0.8em} \ydiagram[*(white) \bullet]
                {1+3,1+1,1+1,1+0} * {1,1,1,1}* {0+0,2+2} \qquad \text{
                  (c) }\quad \ytableausetup{mathmode, boxsize=0.8em}
                \ydiagram[*(white) \bullet] {3+1,1+3,1+1,1+0} *
                         {3,2,1,1}\] 
\caption{A Young diagram, the hook corresponding to $(1,2)$, and the
  associated rim hook.}
\label{young}
\end{figure}

We say that a partition $\lambda$ is an \emph{$r$-staircase} if 
$$\lambda=(r,r-1,\ldots,1).$$ For example, if we remove the box $(2,4)$
from the partition in Figure \ref{young}(a) then we are left with a
$4$-staircase. Note that the maximal rim hook for an $r$-staircase
has length $2r-1$.

Fix a prime number $p$. A \emph{(rim) $p$-hook} is a (rim) hook consisting of
$p$ boxes. The \emph{$p$-core} of a given partition $\lambda$ is the
partition we obtain by successively removing rim $p$-hooks until this
is no longer possible. The $p$-core is independent of the order in
which one removes these $p$-hooks.
Note that $2$-cores are given by $r$-staircases for $r\geq 0$.

 We say that $\lambda$ is
\emph{$p$-restricted} if $\lambda_{i}-\lambda_{i+1} <p$ for all $i$ and
\emph{$p$-regular} if there is no $i$ such that $\lambda_i = \lambda_{i+1} =
\dots = \lambda_{i+p-1}$ (that is if $\lambda^T$ is $p$-restricted).
We will also say that any partition is $0$-regular and $0$-restricted.

To each box $b$ in a Young diagram we associate a corresponding
\emph{content} by setting $\content(b)=j-i$ where $b$ is in row $i$ and
column $j$. When working over some field $k$ we will define the
\emph{residue} $\residue(b)$ associated to $b$ to be the image of 
$\content(b)$ under the standard map from $\ZZ$ to $k$.

For a partition $\lambda \vdash n$, a \emph{standard tableau of shape
$\lambda$} is a numbering of the boxes of the Young diagram of
$\lambda$ with the numbers $1,2, \ldots , n$ in such a way that the
numbers are increasing along the rows and down the columns of
$\lambda$. We denote the set of all standard tableaux of shape
$\lambda$ by $\mathcal{T}_\lambda$.

\subsection*{Representations of the symmetric group}
Let $k$ be a field of characteristic $p\geq 0$. Let ${\rm H}_n$ (or just ${\rm H}$ when this will not cause confusion)
be the group algebra $k \mathfrak{S}_n$. There are several different standard (and cellular) bases for ${\rm H}_n$ available in the literature. Here we will use the Murphy basis introduced in \cite{Mur} and follow the exposition given in \cite[Chapter 3]{Mathas}. For each partition $\lambda \vdash n$ and each pair of standard tableaux $(T_1,T_2)\in \mathcal{T}_\lambda \times \mathcal{T}_\lambda$  Murphy defined an element $m_{T_1,T_2}^\lambda\in {\rm H}_n$.  (In fact these are defined over $\ZZ\mathfrak{S}_n$.) These elements form a standard basis for ${\rm H}_n$.   We will not need the explicit construction  of the Murphy basis elements for this paper and so will only recall some of their properties. 
Following Section 3, for any $\lambda \vdash n$ we define

$${\rm H}^{\unrhd \lambda} = \langle m_{T_1, T_2}^\mu , \, (T_1, T_2)
\in \mathcal{T}_\mu \times \mathcal{T}_\mu, \, \mu\vdash n \, \mbox{with}\, \mu \unrhd \lambda
\rangle,  \quad \mbox{and}$$ 
$${\rm H}^{\rhd \lambda} = \langle m_{T_1, T_2}^\mu , \, (T_1, T_2)
\in \mathcal{T}_\mu \times \mathcal{T}_\mu, \, \mu\vdash n \, \mbox{with}\, \mu \rhd \lambda
\rangle.$$ 
Then we have 
$${\rm H}^\lambda = {\rm H}^{\unrhd \lambda} / {\rm H}^{\rhd \lambda}
\cong S^{\lambda} \otimes (S^{\lambda})^{\rm op}$$ as ${\rm
  H}$-bimodules and for any fixed $T_1\in \mathcal{T}_\lambda$ the set
$m_{T, T_1}^\lambda$ for all $T\in \mathcal{T}_\lambda$ form a basis
for the module $S^\lambda$. In fact, the modules $S^\lambda$ for $\lambda \vdash n$ are the familiar \emph{dual Specht modules} for ${\rm H}_n$.

There is a bilinear form defined on each dual Specht module $S^{\lambda}$
and if we consider the quotient of $S^{\lambda}$ by the radical of its
bilinear form then we have that
$$S^\lambda / {\rm rad}S^{\lambda} \neq 0 \quad \mbox{if and only if}
\quad \mbox{$\lambda$ is $p$-restricted}$$
where $p\geq 0$ is the characteristic of the field $k$. Moreover the set of all
$D^\lambda :=S^\lambda / {\rm rad}S^{\lambda}$ for $\lambda$ running
over the set of $p$-restricted partitions of $n$ form a complete set
of pairwise non-isomorphic \emph{simple} ${\rm H}_n$-modules.

The \emph{blocks} of ${\rm H_n}$ are described by Nakayama's conjecture (see \cite[Theorem 6.1.21]{JK}) which states that two partitions of $n$ are in the same block if and only if they have the same $p$-core. In particular, if a partition $\lambda$  is itself a $p$-core then it is alone in its block and we have $S^\lambda = D^\lambda$.

Define $\iota: {\rm H}_n\rightarrow {\rm H}_n$ to be the anti-automorphism defined by $\iota(x)=x^{-1}$ for all $x\in \mathfrak{S}_n$.
Define also $\alpha : {\rm H}_n\rightarrow {\rm H}_n$ to be the automorphism defined by $\alpha(x)=(-1)^{\ell(x)}x$ for all $x\in \mathfrak{S}_n$. Note that, when restricted to ${\rm H}_n$, the anti-automorphism $\phi$ given in Definition \ref{antiauto} can be factorised as $\phi = \alpha \iota$.

Using the anti-automorphism $\iota$, the linear dual $M^*$ of any left ${\rm H}_n$-module can be given the structure of left ${\rm H}_n$-module by setting $(x\gamma)(m) = \gamma(\iota (x) m)$ for all $\gamma\in M^*$, $m\in M$ and $x\in {\rm H}_n$. This duality fixes every simple module, so we have
$(D^\lambda)^* \cong D^\lambda$ for any $p$-restricted partition $\lambda$ (see \cite[Exercise 2.7]{Mathas}). 

We also have another functor on ${\rm H}_n\text{-mod}$ given by tensoring with the 1-dimensional sign representation ${\rm sgn}$ of ${\rm H}_n$. It is well known that $S^\lambda \otimes {\rm sgn} \cong (S^{\lambda^T})^*$ for any partition $\lambda\vdash n$ (see for example \cite[Exercise 3.14(iii)]{Mathas}). 
For each $p$-restricted partition $\lambda$, the module
$D^{\lambda}\otimes {\rm sgn}$ is also simple and so it must be
isomorphic to some $D^{\mu}$ for some $p$-restricted partition $\mu$. We call the
partition $\mu$ the \emph{Mullineux conjugate} of the partition
$\lambda$, and denoted by $\mu = \lambda^{M}$. The map $-^M$ gives an involution on the set of $p$-restricted partitions of $n$. Note that in the literature, the Mullineux map is often defined as an involution $m$ on the set of $p$-regular partitions. Our Mullineux conjugate $\lambda^M$ is related to the involution $m$ simply by $\lambda^M = (m(\lambda^T))^T$.
There are several combinatorial descriptions of this involution, the first one given by Mullineux in \cite{Mullineux}, but we will not need any explicit description for this paper. We only note that if $S^\lambda = D^\lambda$ then we have $D^\lambda \otimes {\rm sgn} = S^\lambda \otimes {\rm sgn} \cong (S^{\lambda^T})^* \cong D^{\lambda^T}$ and hence $\lambda^M = \lambda^T$ in this case. This happens for instance when $\lambda$ is a $p$-core.

We finish this section with a couple of properties of the Murphy basis which we will need later in the paper.

\begin{proposition}\label{cellular Sn}
The Murphy basis has the following properties. 
\begin{enumerate}
\item $\iota (m_{T_1, T_2}^\lambda) = m_{T_2, T_1}^\lambda$ for all $T_1,T_2\in \mathcal{T}_\lambda, \, \lambda\vdash n$.
\item For each $\lambda \vdash n$ and $T_1\in \mathcal{T}_\lambda$, the set of all $\alpha(m_{T, T_1}^\lambda)$ with $T\in \mathcal{T}_\lambda$ spans an $H_n$-module isomorphic to $(S^{\lambda^T})^*$.
\end{enumerate}
\end{proposition}

\begin{proof}
For part (1), see \cite[Chapter 3, 3.20(1)]{Mathas}. For part (2),
simply note that for any $w\in \mathfrak{S}_n$ we have
$$\alpha(wm_{T,T_1}^\lambda) = \alpha(w)
\alpha(m_{T,T_1}^\lambda) = 
(-1)^{\ell(w)} w \alpha(m_{T,T_1}^\lambda).$$ 
So we have
$$w\alpha(m_{T,T_1}^\lambda) = (-1)^{\ell(w)} \alpha
(wm_{T,T_1}^\lambda).$$ 
Therefore we have
$$\langle \alpha(m_{T,T_1}^\lambda) \, | \, T\in
\mathcal{T}_\lambda\rangle \cong S^{\lambda}\otimes {\rm sgn} \cong
(S^{\lambda^{T}})^*.$$ 
\end{proof}

Note that the set of all $\alpha(m_{T_1,T_2}^\lambda)$ for $\lambda \vdash n, T_1, T_2\in \mathcal{T}_\lambda$ form another standard basis for ${\rm H}_n$.

\section{Standard basis and standard modules for the periplectic Brauer algebra} \label{Section standard modules}

In this section we will continue to consider an arbitrary field $k$ of characteristic $p \geq 0$.
Let $0\leq t\leq n$ with $n-t\in 2 \mathbb{Z}$ and define $I(n,t)$ to be
the set of $(n,t)$-Brauer diagrams with precisely $t$ non-crossing
propagating lines. Following \cite[Section 4]{GL}, each $(n,n)$-Brauer
diagram with $t$ propagating lines can be uniquely written as
$S_1wS_2^{\rm op}$ with $S_1, S_2\in I(n,t)$, and $w\in
\mathfrak{S}_t$, where $S_2^{\rm op}$ denotes the $(t,n)$-Brauer
diagram obtained by flipping $S_2$ horizontally.

\begin{theorem}\label{stdbasis} Set 
$$\Lambda_n = \{\lambda \vdash t, \, 0\leq t\leq n \, \mbox{with}\ \,
  n-t\in 2\mathbb{Z}\}$$  with partial order $\unlhd$.
 Define 
$$\mathcal{B}_n = \{ C_{(S_1,T_1)(S_2,T_2)}^\lambda := S_1
 m_{T_1, T_2}^\lambda S_2^{\rm op}\, | \, \lambda\in \Lambda_n ,\,
 T_1, T_2\in \mathcal{T}_\lambda, \, S_1,S_2\in I(n,t),  \, t=|\lambda|\}.$$  Then
 $(A_n, \mathcal{B}_n)$ is a standardly based algebra.
\end{theorem}

\begin{proof}
The proof follows exactly the proof of (C1) and (C2) in \cite[Theorem
  4.10]{GL}. Note that, up to a sign, the multiplication of diagrams
in $A_n$ is the same as  in the Brauer algebra (with parameter $0$)
and the sign does not affect the arguments as we consider linear spans
of diagrams.
\end{proof}

Following Section \ref{sba} we have a filtration of the periplectic
Brauer algebra $A_n$ with sections
$$A_n^\lambda \cong W_n(\lambda) \otimes W^{\rm op}_n(\lambda)$$ for
each $\lambda \in \Lambda_n$. The left (respectively right) standard module
$W_n(\lambda)$ (respectively $W^{\rm op}_n(\lambda)$) is spanned by the
elements $C_{(S,T)(S_1,T_1)}^\lambda$
(respectively $C_{(S_1,T_1)(S,T)}^\lambda$) for all $S\in I(n,t)$, $T\in
\mathcal{T}_\lambda$ and some fixed $S_1\in I(n,t)$ and $T_1\in
\mathcal{T}_\lambda$.

Note that if we used the basis $\alpha(m_{T_1,T_2}^\lambda)$ instead of $m_{T_1,T_2}^\lambda$ to construct the standard basis for $A_n$, we would get a different set of standard modules, which we denote by $\widetilde{W}_n(\lambda)$, spanned by $\widetilde{C}_{(T,S)(T_1S_1)}^\lambda = S \alpha(m_{T ,T_1}^\lambda) S_1^{op}$.

\begin{remark}\label{standardspecht}
When $|\lambda|=n$ we see that $C_{(S,T)(S_1,T_1)}^\lambda = m_{TT_1}^\lambda$ as $I(n,n)$ only contains the identity element. So in this case we have $W_n(\lambda) = S^\lambda$ inflated to $A_n$ using the surjection $A_n \rightarrow {\rm H}_n$ which maps any diagram with at least one cup (or cap) to $0$.
\end{remark}

More generally, it is clear from the definition of $W_n(\lambda)$ given above that we can write it as
$$W_n(\lambda) = V(n,t)\otimes_{\mathfrak{S}_t} S^\lambda$$
where $V(n,t)$ is the span of all $(n,t)$-Brauer diagrams with exactly $t$ propagating lines and the action of $A_n$ on $V(n,t)\otimes_{\mathfrak{S}_t} S^\lambda$ is given as follows. Let $d$ be an $(n,n)$-Brauer diagram, $S$ be a $(n,t)$-Brauer diagram with $t$ propagating lines and $x\in S^\lambda$. Consider the multiplication $dS$ as defined in (\ref{Composition Brauer diagrams}). If $d \star S$ has fewer than $t$ propagating lines then we set $d(S\otimes x)=0$. Otherwise we set $d(S\otimes x)=dS \otimes x$.

Similarly, using Proposition \ref{cellular Sn} (2), we have that  
$$\widetilde{W}_n(\lambda) = V(n,t)\otimes_{\mathfrak{S}_t} (S^{\lambda^T})^*.$$

In fact, we have that $V(n,t)\otimes_{\mathfrak{S}_t} -$ gives a functor from ${\rm H}_t\text{-mod}$ to $A_n\text{-mod}$. It is easy to see that $V(n,t)$ is projective as a right $\mathfrak{S}_t$-module and so this functor is exact.  Now as $(D^\lambda)^* \cong D^\lambda$ for all simple ${\rm H}_t$-modules we have that $(S^{\lambda^T})^*$ and $S^{\lambda^T}$ have the same composition factors. The next proposition then follows immediately from the exactness of the functor $V(n,t)\otimes_{\mathfrak{S}_t} -$.

\begin{proposition}\label{tildeW}
For any $\lambda \in \Lambda_n$ the $A_n$-modules $\widetilde{W}_n(\lambda)$ and $W_n(\lambda^T)$ have the same composition factors.
\end{proposition}

Recall the definition of the anti-automorphism $\phi$ given in Definition \ref{antiauto}.

\begin{proposition}\label{phicostandard}
Let $\lambda \in \Lambda_n$. We have $$\phi(W^{\rm op}_n(\lambda))
\cong \widetilde{W}_n(\lambda).$$ 
\end{proposition}

\begin{proof}
We have that $\phi(W^{\rm op}_n(\lambda))$ is spanned by the set of
all $\phi(C_{(T_1,S_1)(T,S)}^\lambda)$ for $(T,S)\in
\mathcal{T}_\lambda \times I(n,t)$. Now we have
\begin{eqnarray*}
\phi(C_{(T_1,S_1), (T,S)}^\lambda) &=& \phi(S_1 m_{T_1,T}^\lambda S^{\rm op})\\
&=& \pm S \phi(m_{T_1,T}^\lambda) (S_1)^{\rm op}\\
&=& \pm S \alpha(m_{T, T_1}^\lambda ) (S_1)^{\rm op}\\
&=& \pm \widetilde{C}_{(T,S), (T_1,S_1)}^\lambda).  
\end{eqnarray*} 
Here
we used Proposition \ref{cellular Sn} (1) and the fact that, when restricted to the symmetric group,
$\phi = \alpha \iota$.
These elements span the left $A_n$-module $\widetilde{W}_n(\lambda)$ by definition.
\end{proof}

Using the anti-automorphism $\phi$, we can define a contravariant exact functor \[
\Upsilon\colon A_n\text{-mod}\to A_n\text{-mod}; \quad M \mapsto
M^\ast = \Hom_\mk (M, \mk), \]µ where the $A_n$ action on $M^\ast$ is
given by
\[
a \gamma (m):= \gamma (\phi(a)m), \text{ for all } a\in A_n, \gamma
\in M^\ast, m \in M. 
\]
The functor $\Upsilon$ gives a contravariant equivalence of categories. This
implies the following result. 

\begin{proposition}
Let $\lambda\in \Lambda_n$. All composition factors of
$\Upsilon(W_n(\lambda))$ belong to the same block.
\end{proposition}

\begin{proof}
As $\Upsilon$ is a contravariant equivalence of categories, this
follows directly from the fact that all composition factors of
$W_n(\lambda)$ belong to the same block, as discussed at the end of
Section \ref{sba}.
\end{proof}

\medskip

Now define
$$\Lambda_n' = \{\lambda \in \Lambda_n \, : \, \lambda \, \mbox{is
  $p$-restricted and $\lambda \neq \emptyset$ when $n$ is even}\}.$$  It was shown in \cite{Coulembier} that the set
of $L_n(\lambda):=W_n(\lambda) / {\rm rad} W_n(\lambda)$ for all
$\lambda\in \Lambda_n'$ 
form a complete set of pairwise non-isomorphic 
simple $A_n$-modules. For
each $\lambda \in \Lambda_n'$ we denote by 
$P_n(\lambda)$ the projective cover of $L_n(\lambda)$.

\medskip

We will later wish to consider reduction mod $p$ from characteristic
$0$ to characteristic $p$. Notice that the basis in Theorem
\ref{stdbasis} is defined over $R$, and we will later use this to
relate the standard modules over $\mK$ to those over $\mk$. When we
wish to emphasise the choice of field we will use a superscript; for
example $W^{\mK}_n(\lambda)$ will denote the standard module defined
over $\mK$.

\section{Localisation and globalisation functors}\label{Section embeddings}

In this section we work over an arbitrary field $k$. Recall that the
symmetric group algebra ${\rm H}_n$ is the 
quotient of $A_n$ with basis  consisting of Brauer diagrams
without cups or caps. We can extend every ${\rm H}_n$-module
$M$ to an $A_n$-module by letting the Brauer diagrams which contains
cups and caps act trivially.  This gives an embedding of the category
${\rm H}_n \text{-mod}$ into the category $A_n\text{-mod}$.

We can also embed $A_n\text{-mod}$ in $A_{n+2}\text{-mod}$ for $n\geq 1$. For this
we need the idempotent $\epsilon_{n+2}$ which has $n-1$ non-crossing
propagating lines joining the left-most nodes, a propagating line
joining the third northern node from the right with the right-most
southern node, one cup joining the two right-most northern nodes and
one cap joining the penultimate southern node with its left neighbour,
as illustrated in Figure \ref{idemis}.

\begin{figure}[ht]
\[
\begin{tikzpicture}[scale=1,thick,>=angle 90]
\begin{scope}
\draw (-2,0.5) node[]{$\epsilon_{n+2}=$};

\draw (0,0)--(0,1);
\draw (1,0)--(1,1);
\draw (2,0.5) node[]{$\dots$};
\draw (3,0)--(3,1);
\draw (4,1) -- (6,0);
\draw (4,0) to [out=90,in=180] +(.5,.4) to [out=0,in=90] +(.5,-.4);
\draw (5,1) to [out=-90, in = 180] +(.5,-.4) to [out = 0, in =-90] +(0.5,0.4);
\end{scope}
\end{tikzpicture}.
\]
\caption{An idempotent in $A_{n+2}$}
\label{idemis}
\end{figure}
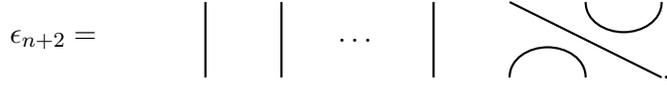

We also introduce the $(n+2,n)$-Brauer
diagram $g_{n+2,n}$ which has $n$ non-crossing propagating lines
connecting the left-most nodes and one cup connecting
the remaining northern nodes and the $(n,n+2)$-Brauer diagram $f_{n,n+2}$
which has $n-1$ non-crossing propagating lines connecting the
left-most nodes, one propagating line connecting the right-most nodes
and one cap connecting the remaining southern nodes. We
illustrate these two diagrams in Figures \ref{gis} and \ref{fis}. It is clear
that $f_{n,n+2}g_{n+2,n}= \id_n$ and $g_{n+2,n} f_{n,n+2}=
\epsilon_{n+2}$.  

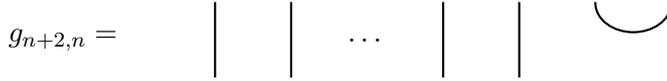
\begin{figure}[ht]
\begin{tikzpicture}[scale=1,thick,>=angle 90]
\begin{scope}
\draw (-2,0.5) node[]{$g_{n+2,n}=$};
\draw (0,0)--(0,1);
\draw (1,0)--(1,1);
\draw (2,0.5) node[]{$\dots$};
\draw (3,0)--(3,1);
\draw (4,0)--(4,1);
\draw (5,1) to [out=-90,in=180] +(0.5,-0.4) to [out=0, in =-90] +(0.5,0.4);
\end{scope}
\end{tikzpicture}
\caption{The $(n+2,n)$-Brauer diagram $g_{n+2,n}$}
\label{gis}
\end{figure}

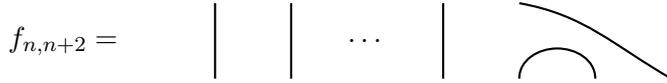
\begin{figure}[ht]
\begin{tikzpicture}[scale=1,thick,>=angle 90]
\begin{scope}
\draw (-2,0.5) node[]{$f_{n,n+2}=$};
\draw (0,0)--(0,1);
\draw (1,0)--(1,1);
\draw (2,0.5) node[]{$\dots$};
\draw (3,0)--(3,1);
\draw (4,0) to [out=90,in=180] +(0.5,0.4) to [out=0, in =90] +(0.5,-0.4);
\draw (6,0) to [out=150, in=-10] +(-2,1);
\end{scope}
\end{tikzpicture}
\caption{The $(n,n+2)$-Brauer diagram $f_{n,n+2}$}
\label{fis}
\end{figure}

One can then easily verify the following lemma.
\begin{lemma}\label{embedding An into An+2}
Let $n\geq 1$. There is an algebra isomorphism
\[
A_n \to \epsilon_{n+2} A_{n+2} \epsilon_{n+2},
\]
which maps $a\in A_n$ to $g_{n+2,n} a f_{n,n+2}$. The inverse maps
$b\in \epsilon_{n+2} A_{n+2} \epsilon_{n+2}$ to $f_{n,n+2} b
g_{n+2,n}$.
\end{lemma}

We define a functor $F_{n+2} \colon A_{n+2}\text{-mod} \to
A_n\text{-mod}$ by mapping $M$ to $\epsilon_{n+2}M$ and using the
isomorphism of Lemma \ref{embedding An into An+2}. We also define
\begin{align*}
G_n \colon  A_{n}\text{-mod} & \to A_{n+2}\text{-mod} \\
 M & \mapsto A_{n+2} \epsilon_{n+2}\otimes_{A_n} M.
\end{align*}
It is clear that $F_{n+2} G_{n}(M) \cong M$ for all $A_{n}$-modules
$M$. Hence $G_n$ gives an embedding of $A_n\text{-mod}$ in
$A_{n+2}\text{-mod}.$

The functors $F_n$ and $G_n$ as defined above are analogues of
corresponding functors for the ordinary Brauer algebra considered in
\cite{CDM} and \cite{DWH}. As vector spaces, the standard modules
$W(\lambda)$ for the periplectic Brauer algebra are isomorphic to the
corresponding standard modules $\Delta(\lambda)$ for the Brauer
algebra, and so it is easy to see that the arguments in \cite{CDM} and \cite{DWH} generalise to the periplectic case to give the following
result.

\begin{lemma}\label{embedding cell modules}
For $n\geq 1$ and $\lambda \in \Lambda_n$ we have
\[
G_n (W_n(\lambda)) = W_{n+2}(\lambda).
\]
\end{lemma}

It follows immediately that for $n\geq 3$ and $\lambda\in \Lambda_n$ we have
$$F_n(W_n(\lambda))\cong\left\{\begin{array}{ll}
W_{n-2}(\lambda) & \text{if}\ \lambda\in\Lambda_{n-2}\\
0& \text{otherwise}\end{array}\right.$$
and from the exactness of $F_n$ that for $\lambda\in \Lambda_n'$ we have
$$F_n(L_n(\lambda))\cong\left\{\begin{array}{ll}
L_{n-2}(\lambda) & \text{if}\ \lambda\in\Lambda_{n-2}'\\
0& \text{otherwise.}\end{array}\right. $$
By induction and exactness of $F_n$ we obtain the following result.

\begin{lemma}\label{Lemma reduce decomposition numbers}
Let $\lambda \in \Lambda_n$ and $\mu \in \Lambda_n'$.  Then
\[
[W_n(\lambda): L_n(\mu)]= \begin{cases} [W_{\abs{\mu}}(\lambda):
    L_{\abs{\mu}}(\mu)] &\text{ if } \abs{\lambda} \leq \abs{\mu},
  \\ 0 &\text{ if }\abs{\lambda} > \abs{\mu}. \end{cases} 
\]
\end{lemma}

\section{BGG reciprocity in arbitrary characteristic}
\label{Sec BGG}
We continue to work over an arbitrary field $k$ of characteristic $p\geq 0$. In the last Section we saw
that certain aspects of the representation 
theories of the periplectic and Brauer algebras are very similar. In
contrast, the results in this Section begin to illustrate their
striking differences. 

\begin{proposition}\label{upLis}
For any $\lambda \in \Lambda_n'$ we have
$$\Upsilon(L_n(\lambda)) \cong L_n(\lambda^M)$$
where $\lambda^M$ denotes the Mullineux conjugate of the partition $\lambda$.
\end{proposition}

\begin{proof}
We prove this by induction on $n$. If $n=0$ or $1$ there is nothing to
prove as $A_0 = A_1 = \mk$ and there is only one simple module. Let
$n\geq 2$. If $|\lambda | = n$ then $L_n(\lambda) = D^\lambda$ (lifted
to $A_n$). In this case we have $\Upsilon(D^\lambda) = (D^\lambda
\otimes {\rm sgn})^* = (D^{\lambda^M})^* \cong D^{\lambda^M}$ so we are
done.

If $|\lambda |\leq n-2$ then $\Upsilon(L_n(\lambda))$ is certainly a
simple $A_n$-module.
Applying the localisation functor, we get the $A_{n-2}$-module $F_n(\Upsilon(L_n(\lambda)) = \epsilon_n \Upsilon(L_n(\lambda))$, where we use the isomorphism $A_{n-2} \cong \epsilon_n A_n \epsilon_n$ from Lemma \ref{embedding An into An+2}.  We have $$\epsilon_n \Upsilon(L_n(\lambda)) = \Upsilon(\phi(\epsilon_n)L_n(\lambda)).$$
Now $\phi(\epsilon_n)L_n(\lambda)$ is a simple $\phi(\epsilon_n) A_n\phi(\epsilon_n)$-module and $\phi(\epsilon_n) A_n\phi(\epsilon_n)\cong A_{n-2}$ (this can be seem by swapping the roles of $f$ and $g$ in Lemma \ref{embedding An into An+2}).  Using the corresponding localisation functor we get $\phi(\epsilon_n)L_n(\lambda) \cong L_{n-2}(\lambda)$.   
Finally, we obtain
$$F_n(\Upsilon(L_n(\lambda)) \cong \Upsilon(L_{n-2}(\lambda)) \cong L_{n-2}(\lambda^M)$$ by induction. 
Thus we must have
$\Upsilon(L_n(\lambda)) \cong L_n(\lambda^M)$ as required.
\end{proof}

\begin{lemma}\label{lemmaproj}
Let $\lambda\in \Lambda_n'$ and $e_\lambda$ be a primitive idempotent
satisfying $P_n(\lambda)=A_n e_\lambda$. Then we have
$$A_n\phi(e_\lambda) = P_n(\lambda^M).$$
\end{lemma}

\begin{proof}
Clearly $\phi(e_\lambda)$ is a primitive idempotent. So
$A_n\phi(e_\lambda) = P_n(\mu)$ for some $\mu \in \Lambda_n'$. Now we
have
\begin{eqnarray*}
\Hom_{A_n} (P_n(\mu), L_n(\lambda^M)) &\cong& \Hom_{A_n}(
A_n\phi(e_\lambda), \Upsilon(L_n(\lambda))) \\ 
&\cong& \phi(e_\lambda)\Upsilon(L_n(\lambda)) \\
&=& \Upsilon(e_\lambda L_n(\lambda))\\
&\cong & \Upsilon(\Hom_{A_n}(A_ne_\lambda , L_n(\lambda))\\
&=& \Upsilon (\Hom_{A_n}(P_n(\lambda), L_n(\lambda))\\
&=& \Upsilon (\mk) \cong \mk.
\end{eqnarray*}
This shows that we must have $\mu = \lambda^M$ as required.
\end{proof}

The following theorem generalises the BGG-reciprocity given in \cite[Theorem 3]{Coulembier} to field of arbitrary characteristics.

\begin{theorem}\label{bgg}
Let $\lambda \in \Lambda_n'$. There is a filtration of the projective indecomposable module $P_n(\lambda)$ by standard modules and if we denote by $(P_n(\lambda): W_n(\mu))$ the number of subquotients in this filtration which are isomorphic to $W_n(\mu)$ for $\mu\in \Lambda_n$ then we have
$$(P_n(\lambda): W_n(\mu)) = [W_n(\mu^T): L_n(\lambda^M)]$$ where
$\mu^T$ denotes the transpose of the partition $\mu$ and $\lambda^M$
denotes the Mullineux conjugate of the partition $\lambda$.
\end{theorem}

 (Note that when $p>n$ or $p=0$ we have that $\Lambda_n' = \Lambda_n \backslash \{ 	\varnothing\}$ and $\lambda^M = \lambda^T$ for all $\lambda\in \Lambda_n$.) 

\begin{proof}
Using Proposition \ref{standardproperties} (2) we have that
$$(P_n(\lambda): W_n(\mu)) = \dim (W_n^{\rm op}(\mu) \otimes_{A_n} P_n(\lambda)).$$
Now we have
\begin{eqnarray*}
\dim (W_n^{\rm op}(\mu) \otimes_{A_n} P_n(\lambda)) &=& \dim (W_n^{\rm
  op}(\mu) \otimes_{A_n} A_n e_\lambda) \\ 
&=& \dim W_n^{\rm op}(\mu) e_\lambda \\
&=& \dim \phi(W_n^{\rm op}(\mu)e_\lambda)\\
&=& \dim \phi(e_\lambda) \phi (W_n^{\rm op}(\mu))\\
&=& \dim \phi(e_\lambda) \widetilde{W}_n(\mu)
\end{eqnarray*}
using Proposition \ref{phicostandard}.  Then we get
\begin{eqnarray*}
\dim \phi(e_\lambda) \widetilde{W}_n(\mu) &=& \dim \Hom_{A_n}(A_n\phi(e_\lambda), \widetilde{W}_n(\mu) \\
&=& \dim \Hom_{A_n} (P(\lambda^M), \widetilde{W}_n(\mu)) \\
&=& [\widetilde{W}_n(\mu): L_n(\lambda^M)]
\end{eqnarray*}
using Lemma \ref{lemmaproj}. Finally, using Proposition \ref{tildeW}, we get that
$$[\widetilde{W}_n(\mu): L_n(\lambda^M)] = [W_n(\mu^T): L_n(\lambda^M)]$$
as required.
\end{proof}

As a consequence of this BGG-reciprocity, we obtain the following factorisation of the Cartan matrix of the periplectic Brauer algebra $A_n$ in arbitrary characteristic.

\begin{corollary}\label{cartan}
For $\lambda, \nu\in \Lambda_n'$ and $\mu\in \Lambda_n$ define the composition multiplicities
$$C_{\lambda \nu} = [P_n(\lambda): L_n(\nu)] \quad \mbox{and} \quad D_{\mu \nu} = [W_n(\mu): L_n(\nu)].$$
Then we have
$$C_{\lambda \nu} = \sum_{\mu\in \Lambda_n} D_{\mu^T \lambda^M} D_{\mu \nu}.$$
\end{corollary}

\section{Blocks in characteristic $p>2$} \label{Sec blocks}

Recall from Section 3 that we can define the blocks of $A_n$ as equivalence classes on $\Lambda_n$.
In \cite{Coulembier}, Coulembier described the blocks of the periplectic Brauer algebra $A_n$ when the characteristic $p$ of the field satisfies $p\notin [2,n]$. We recall his result below.

\begin{theorem} \cite[Theorem 1]{Coulembier}
Let $A_n$ be the periplectic Brauer algebra over a field of characteristic $p\notin [2,n]$. 
Let $\lambda, \mu\in\Lambda_n$. Then $\lambda$ and $\mu$ are in the same $A_n$-block if and only if they have the same $2$-core. 
\end{theorem}

The aim of this section is to generalise this result to fields of characteristic $p>2$. As described in the Introduction we will fix a $p$-modular system $(\mK, R, \mk)$
with ${\rm char }\; \mk = p>2$. Throughout this section we consider the periplectic Baruer algebra $A_n$ over the field $\mk$.


\begin{proposition} \label{same p-core}
Let $\lambda, \mu\in \Lambda_n$. If $|\lambda|=|\mu|$ and $\lambda$ and $\mu$ have the same $p$-core then $\lambda$ and $\mu$ are in the same $A_n$-block.
\end{proposition}
\begin{proof}
We know that if $\lambda$ and $\mu$ have the same $p$-core, then
$S^\lambda$ and $S^\mu$ belong to the same $\mk
\mathfrak{S}_{|\lambda|}$-block  by Nakayama's conjecture \cite[Theorem
  6.1.21]{JK}. 
 Using Remark \ref{standardspecht} we deduce that $W_{|\lambda|}(\lambda)$ and $W_{|\lambda|}(\mu)$ belong to the same $A_n$-block. Now, by repeated application of the globalisation functors $G_{n-2}\ldots G_{|\lambda|+2} G_{|\lambda|}$ and Lemma \ref{embedding cell modules} we deduce that $W_n(\lambda)$ and $W_n(\mu)$ belong to the same $A_n$-blocks.
\end{proof}

\begin{proposition}\label{transblock}
Let $\lambda$ and $\mu$ be 
in $\Lambda_n$. Then we have that $\lambda$ and $\mu$ are in
the same $A_n$- block if and only if $\lambda^T$ and $\mu^T$ are in
the same $A_n$-block.
\end{proposition}

\begin{proof}
First suppose that $\lambda$ and $\mu$ are in $\Lambda'_n$.
The simple modules $L_n(\lambda)$ and $L_n(\mu)$ are in the same block
if and only if $\Upsilon(L_n(\lambda))$ and $\Upsilon(L_n(\mu))$ are
in the same block. But we have $\Upsilon(L_n(\lambda))\cong
L_n(\lambda^M)$ and $\Upsilon(L_n(\mu)) \cong L_n(\mu^M)$ by
Proposition \ref{upLis}.  Now for any $p$-restricted partition $\nu$
we have that $\nu^T$ and $\nu^M$ belong to the same block. This
follows from the fact that $S^\nu \otimes {\rm sgn} \cong
(S^{\nu^T})^*$ and $D^\nu \otimes {\rm sgn} \cong D^{\nu^M}$, together
with Corollary \ref{stdinP} and the fact that duality fixes simple
modules for the symmetric 
group.  Thus we can conclude that $\lambda$ and $\mu$ belong to the
same block if and only if $\lambda^T$ and $\mu^T$ belong to the same
block.

Now suppose that $\lambda$ and $\mu$ are general. Choose $\lambda'$
and $\mu'$ in $\Lambda_n'$ such that $\lambda'$ has the same $p$-core as
$\lambda$ (respectively $\mu'$ has the same $p$-core as $\mu$), and
$|\lambda'|=|\lambda|$ (respectively $|\mu'|=|\mu|$). These must exist
by considering composition factors of the associated Specht modules
and using Proposition \ref{same p-core}. Further we have that
$(\lambda')^T$ has the same $p$-core as $\lambda^T$ and $(\mu')^T$ has
the same $p$-core as $\mu^T$. Therefore $\lambda$ is in the same block
as $\lambda'$ and $\lambda^T$ is in the same block as $(\lambda')^T$
(and similarly for $\mu$) by Proposition \ref{same p-core}. The result
now follows from the first case considered above.
\end{proof}

\begin{proposition}\label{removing horizontal two hook}
If $\lambda$ can be obtained from $\mu\in \Lambda_n$ by
removing two boxes in the same row (respectively column), then $\lambda$
and $\mu$ belong to the same $A_n$- block.
\end{proposition}
\begin{proof}
By Proposition \ref{transblock} it is enough to consider the case when
$\lambda$ is obtained from $\mu$ by removing two boxes in the same
row.  For a field $\mK$ of characteristic zero, we have
\cite[Proposition 7.2.6] {Coulembier}
\[
[W_n^\mK(\lambda) : L_n^\mK(\mu)]=1.
\]
This implies that there exists a submodule $M$ of $W_n^\mK(\lambda)$ such that 
\[
\Hom_{A_n^\mK}(W_n^\mK(\mu), W_n^\mK(\lambda)/M)\not=0. 
\]
By \cite[Lemma 4.1]{King}, we can reduce this modulo $p$ to obtain 
\[
\Hom_{A_n^\mk}(W_n^\mk(\mu), W_n^\mk(\lambda)/\overline{M})\not=0, 
\]
where $\overline{M}$ is a submodule of $W_n^\mk(\lambda)$.
In particular $\mu$ and $\lambda$ belong to the same block. 
\end{proof}

We say that $\ft= (\ft^{(1)},\ft^{(2)},\dots,\ft^{(n)})$ is a \emph{path of
partitions} if each $\ft^{(i)}$ is a partition such that
$\ft^{(i+1)}$ is obtained from $\ft^{(i)}$ by adding or removing one
box in the Young diagram and $\ft^{(1)}=(\ytableausetup{mathmode,
  boxsize=0.6em} \ydiagram{1})$. We denote the set of all paths of
length $n$ with $\ft^{(n)} = \lambda$ by $St_n(\lambda)$.  We define the
vector $c_\ft=(c_\ft(2),c_\ft(3),\dots, c_\ft(n)) \in \mk^{n-1}$ for
$\ft \in St_n(\lambda)$ as follows
\[
c_\ft(i) = \begin{cases} \residue(b) &\text{ if } \ft^{(i)}= \ft^{(i-1)}
  \cup b \\ \residue(b)+1 &\text{ if } \ft^{(i)}\cup b=
  \ft^{(i-1)},  \end{cases} 
\]
where $b$ is the box added to or removed from $\ft^{(i-1)}$ to obtain $\ft^{(i)}$.

We will need the following pair of lemmas.

\begin{lemma}[{\cite[Proposition 6.2.6]{Coulembier}}] \label{Prop Coulembier}
For $\lambda \in \Lambda_n$ and $\mu \in \Lambda_n'$, if $[W_n(\lambda):
  L_n(\mu)] \not=0$ then there exist $\mathfrak{t}\in St_n(\lambda)$
and $\mathfrak{s} \in St_n(\mu)$ such that $c_\ft=c_\mathfrak{s}$.
\end{lemma}

\begin{lemma}\label{Lemma p-core}
Assume $\lambda \in \Lambda_n$ has as $2$-core an $r$-staircase with 
$2r-1<p$, and that
$$\frac{r(r+1)}{2}+p-2r >n.$$ Then $\lambda$ is a
$p$-core and every partition $\mu$ obtained by adding a box to $\lambda$ is
still a $p$-core. Further, two boxes in $\mu$ (or in $\lambda$) have
the same residue if 
and only if they have the same content.
\end{lemma}
\begin{proof}
The degree of an $r$-staircase is $\frac{r(r+1)}{2}$ and the length of its rim is $2r-1$. The partition
$\lambda$ is obtained by adding at most $p-2r-1$ boxes to the
$r$-staircase since $\frac{r(r+1)}{2}+p-2r >n$. Therefore the length
of the rim of $\lambda$ is at most $2r-1+p-2r-1 = p-2$.  In particular
$\lambda$ is a $p$-core and this is still true for the partition $\mu$
since the length of the rim will still be smaller than $p$. As the
possible contents occuring in $\mu$ all occur in the rim, and these
boxes all have different residues, the residues in $\mu$ for boxes with
differing contents must be distinct.
\end{proof}

\begin{proposition}\label{prop decomposition numbers non-zero}
Consider $\mu \in \Lambda_n'$ such that the $2$-core of $\mu$ is given
by an r-staircase with $2r-1<p$ and $\frac{r(r+1)}{2}+p-2r >n$. Then
for all $\lambda \in \Lambda_n$ we have that 
\[
[W_n(\lambda): L_n(\mu)] \not=0
\]
implies $\lambda \subseteq \mu$ and $\lambda$ and $\mu$ have the same $2$-core, 
and for all $\nu \in \Lambda_n'$, 
\[
[W_n(\mu): L_n(\nu)]\not=0
\]
implies $\mu \subseteq \nu$ and $\mu$ and $\nu$ have the same $2$-core.
\end{proposition}
\begin{proof}

Assume $[W_n(\lambda): L_n(\mu)] \not=0$ for $\lambda \in \Lambda_n$. We may
further assume $n=\abs{\mu}$ by Lemma \ref{Lemma reduce decomposition
  numbers}. From Lemma \ref{Prop Coulembier}, it then follows that
there exists $\ft \in St_n(\lambda)$ and $\mathfrak{s} \in St_n(\mu)$
such that $c_\ft=c_\mathfrak{s}$. Because $\mu \vdash n$, we only add
boxes in $\mathfrak{s}$ and $c_\mathfrak{s}(i)$ is equal to the
residue of the added box.

If $\lambda \not \subseteq \mu$, then there exists a box $b$ in
$\lambda$ which is not contained in $\mu$ but such that $\mu\cup b$ is
a partition. The residue of $b$ appears in $c_\ft$ and thus also in
$c_\mathfrak{s}$. Therefore $\mu$ should also contain a box $b'$ not
in $\lambda$ with the same residue. But from Lemma \ref{Lemma p-core}
it follows that $b$ and $b'$ must have the same content, which is
clearly impossible as they would have to lie on the same diagonal and
$b'$ would then belong to $\lambda$ (since $\lambda$ is a
partition). So we conclude that $\lambda \subseteq \mu$. Moreover,
since the residues determines the contents for $\lambda$ and $\mu$ (by
Lemma \ref{Lemma p-core}) it follows in the same way as in the
characteristic zero situation that $\mu$ and $\lambda$ have the same
$2$-core, see \cite[Corollary 6.2.7 and Lemma 7.3.3]{Coulembier}.

Now assume $[W_n(\mu): L_n(\nu)]\not=0$ for $\nu \in \Lambda_n'$. By
Lemma \ref{Lemma reduce decomposition numbers}, it again suffices to
consider the case $n=\abs{\nu}$. If $\abs{\mu}=\abs{\nu}$ it follows
as in the previous case that the existence of a box in $\nu$
not contained in $\mu$ is impossible. So then $\mu=\nu$.

Consider $\abs{\mu}<\abs{\nu}$. We have $\ft \in St_n(\nu)$ and
$\mathfrak{s} \in St_n(\mu)$ such that $c_\ft=c_\mathfrak{s}$ as
follows from Lemma \ref{Prop Coulembier}. Let $\kappa$ be the
partition containing all the boxes which are added in
$\mathfrak{s}$. In particular $\kappa$ contains $\mu$ and the $2$-core
$(r,r-1,\dots, 2,1)$ of $\mu$. If we would need to add more than
$p-2r-2$ boxes to obtain $\kappa$ from this $2$-core, then $\abs{\nu}
\geq \abs{\kappa}+1> p-2r-2+\frac{r(r+1)}{2}+1\geq n$. This is
impossible since $\nu \vdash n$. Thus we add at most $p-2r-2$ boxes to
the $r$-staircase to obtain $\kappa$ and the length of the rim of
$\kappa$ is smaller than or equal to $2r-1+p-2r-2=p-3$.

We claim that $\nu_1 \leq \kappa_1+1$ and that $\kappa_k=0$ implies
$\nu_{k+1}=0$.  This can be seen as follows. If $\nu_1 > \kappa_1+1$,
then $\nu$ would contain a box in the first row which is two places to
the right of the rightmost box of $\kappa$. The residue of this box
is not equal to $\residue(b)$ or $\residue(b)+1$ for any box $b$ in $\kappa$
because the length of the rim of $\kappa$ is smaller than $p-2$. This
is in contradiction with $c_\ft=c_\mathfrak{s}$.  Similarly there can
be no box of $\nu$ two rows under the last non-zero row of $\kappa$.
These conditions on $\nu$ mean that the length of the rim of $\nu$ is
smaller than $p$. Hence the contents of $\nu$ are determined by the
residues. Then $\mu\subseteq \nu$ again implies that they have the
same $2$-core and $\mu \not \subseteq \nu$ is impossible because it
would mean that $\mu$ contains a box with residue not occurring in
$\nu$.
\end{proof}

For each $\lambda\in \Lambda_n$, denote by $\Lambda_n(\lambda)$ the subset of partitions of $\Lambda_n$ which
are in the same block as
$\lambda$. 

\begin{proposition}\label{r-staircase block}
Consider the r-staircase partition $\rho_r =(r,r-1,r-2,\dots,2,1)$ in
$\Lambda_n$ where $r$ is such that $2r-1<p$ and
$\frac{r(r+1)}{2}+p-2r >n$. Then $\lambda \in \Lambda_n(\rho_r)$ if and
only if the $2$-core of $\lambda$ is $\rho_r$.
\end{proposition}
\begin{proof}
Assume that the $2$-core of $\lambda \in \Lambda_n$ is $\rho_r$. Then we can
find a chain of partitions \[ \rho_r=
\lambda^{(0)}\subset\lambda^{(1)} \subset \dots \subset
\lambda^{(l)}=\lambda, \] such that each $\lambda^{(i)}$ is obtained
from $\lambda^{(i-1)}$ by adding a vertical or horizontal
$2$-hook.  Then it follows from Proposition
\ref{removing horizontal two hook} that $\lambda$ and
$\rho_r$ belong to the same block.

Now consider an arbitrary $\mu$ in $\Lambda_n'$ with the $2$-core of $\mu$
not equal to $\rho_r$. We will show that \[ [ P_n(\lambda):L_n(\mu) ]
=0 \text{ and } [ P_n(\mu):L_n(\lambda) ]=0,
\] 
for all $\lambda \in \Lambda_n'$ which have as $2$-core $\rho_r$. This
implies that 
$$\Lambda_n(\rho_r)= \{\lambda \in \Lambda_n \mid \text{t	he $2$-core of 
$\lambda$ is } \rho_r\}.$$  Using Corollary \ref{cartan}, we find 
$$
[ P_n(\lambda):L_n(\mu) ] =
 \sum_{\gamma \in \Lambda_n}  [W_n(\gamma^T) :
  L_n(\lambda^M)][W_n(\gamma): L_n(\mu) ]. 
$$
Observe first that $\lambda^M=\lambda^T$, since $\lambda$ is a $p$-core by Lemma \ref{Lemma p-core}. Therefore $\lambda^M$  also has $\rho_r$ as $2$-core. 
We know that $[W_n(\gamma^T) : L_n(\lambda^M)]$ is non-zero only if
$\gamma^T$ has the same $2$-core as $\lambda^M$ from Proposition
\ref{prop decomposition numbers non-zero}. This implies that $\gamma$
has $\rho_r$ as $2$-core. But then $ [W_n(\gamma): L_n(\mu) ]$ is
non-zero only if $\gamma$ and $\mu$ have the same $2$-core, which is
impossible since the $2$-core of $\mu$ is not equal to $\rho_r$. So we
conclude $[ P_n(\lambda):L_n(\mu) ]=0$. In a similar way we find $[
  P_n(\mu):L_n(\lambda) ]=0.$ This shows that $\mu$ is not in the same block as $\lambda$. 
Note that we also have that any $\mu\in \Lambda_n$ (not necessarily in $\Lambda_n'$) with $2$-core different from $\rho_r$ would also be in a different block from $\lambda$ using Proposition \ref{prop decomposition numbers non-zero}.
\end{proof}

\begin{proposition}\label{One block}
Consider $\lambda \in \Lambda_n$ which has as $2$-core an $r$-staircase with
$2r-1\geq p$ or $ \frac{r(r+1)}{2}+p-2r \leq n$.  Then
$\lambda$ belongs to the same block as the empty partition $\varnothing$ if $n$
is even or to the same block as the partition $(1)$ if $n$ is odd.
\end{proposition}
\begin{proof}
We will prove the proposition using induction on the number of boxes
in the Young diagram corresponding to the partition. The induction
base is trivially satisfied. We will now show that $\lambda$ is
contained in the same block as a
partition with contains two boxes fewer than $\lambda$ and which has
as $2$-core an $r' $-staircase which also satisfies $2r'-1\geq p$ or $
\frac{r'(r'+1)}{2}+p-2r' \leq n$. By the induction hypothesis, it
then follows that $\lambda$ is contained in the same block as
$\varnothing$ if $n$ is even or as $(1)$ if $n$ is odd.

First suppose that $\lambda$ is not a $2$-core. Then there is a
horizonal or vertical $2$-hook which can be removed to leave a new
partition $\mu$ with the same $2$-core. By Proposition \ref{removing
  horizontal two hook} $\lambda$ and $\mu$ lie in the same block and $\mu$ has the same $2$-core as $\lambda$.

Next suppose that $\lambda$ is a $2$-core, but is not a $p$-core. Then
$\lambda$ is an $r$-staircase and hence any removable rim $p$-hook
cannot lie entirely in the first row. As a removable rim $p$-hook must
exist, we can construct a new partition $\mu$ by removing from
$\lambda$ this rim $p$-hook and adding $p$ boxes to the first
row. Then $\lambda$ and $\mu$ belong to the same block by
Proposition \ref{same p-core} since $\lambda$ and $\mu$ have the same
$p$-core and $\abs{\lambda}=\abs{\mu}$. It is also clear that we can
remove a horizontal $2$-hook from the first row of $\mu$ to obtain a new partition
$\nu$. Then Proposition \ref{removing horizontal two hook} implies
that $\lambda$ and $\nu$ are in the same block. Moreover,
since $\mu$ has a removable rim $p$-hook, $\mu$, and thus also $\nu$,
can not have a $2$-core with $2r-1<p$ and $\frac{r(r+1)}{2}+p-2r >n$
because these conditions would imply that $\mu$ is a $p$-core by Lemma
\ref{Lemma p-core}.

The only case we did not cover yet is when $\lambda$ is a $p$-core and
a $2$-core at the same time. Since $\lambda$ is a $2$-core it is an
$r$-staircase, i.e. $\lambda = (r, r-1, r-2, \dots, 2,1)$.  Because
$\lambda$ is also a $p$-core it follows that $2r-1 <p$ and thus
$2r\leq p-1$ since $p$ is odd. Our condition on $\lambda$ then implies
that $\frac{r(r+1)}{2}+p-2r \leq n.$ Note that this inequality is actually strict, since the parity of the left hand side is different from the parity of the right hand side. This follows  because $\rho_r$ is contained in $\Lambda_n$, and therefore $\frac{r(r+1)}{2}$ and $n$ have the same parity. We can thus add $p-2r+1$ boxes to
the first row of $\lambda$ to obtain a new partition $\mu \in \Lambda_n$.

Now $\lambda$ and $\mu$ are in the same block
by Proposition \ref{removing horizontal two hook}.  Furthermore $\mu$
has a removable rim $p$-hook (consisting of all boxes on the rim of $\mu$) and the same $p$-core as the partition
$\nu = (r-2+p, r-3, r-4, \dots, 1)$ obtained from $\mu$ by removing
the rim $p$-hook and adding $p$ boxes to the first row. Then
$\mu$ and $\nu$ belong to the same block since they have the
same $p$-core and $\abs{\mu}=\abs{\nu}$, while $\nu$ belongs to
the same block as $\kappa$ where $\kappa=(3r-5, r-3, r-4, \dots,
1)$ is obtained from $\nu$ by removing $p-2r+3$ boxes in the first
row. We conclude that $\lambda$ is in the same block as
$\kappa$ and $\abs{\kappa}= \abs{\lambda}-2$.  Furthermore, from
Lemma \ref{Lemma p-core} and the fact that $\nu$ is not a $p$-core, it
follows that the $2$-core of $\nu$ and thus also of $\kappa$ satisfies
the conditions on $r$.

Hence for every $\lambda$ with $\abs{\lambda} \geq 2$ there exists a
$\mu$ such that $\abs{\mu}= \abs{\lambda}-2$ and $\lambda$ is
contained in the same block as $\mu$ and for which the $2$-core
of $\mu$ is an $r$-staircase which satisfies $2r-1\geq p$ or $
\frac{r(r+1)}{2}+p-2r \leq n$.
\end{proof}

For each $\lambda\in \Lambda_n'$, denote by $B_n(\lambda)$ the block algebra of $A_n$ containing the simple module $L_n(\lambda)$.

\begin{theorem} \label{Block decomposition}
The block decomposition of $A_n$ is given by
\[
B_n(\kappa) \oplus \bigoplus_{r} B_n(\rho_r),
\]
where the sum is over all $r\geq 2$ such that $2r-1<p$,
$\frac{r(r+1)}{2}+p-2r >n$ and $\frac{r(r+1)}{2}\vdash n-2k$ for some
$k\geq 0$. Here
$\rho_r$ is the $r$-staircase partition and
$\kappa=(1)$ if $n$
is odd or $\kappa = (1,1)$ if $n$ is even.
 
In particular if $n\geq (p^2+7)/8$, there is only one block.
\end{theorem}
\begin{proof}
The block decomposition follows immediately by combining Proposition \ref{r-staircase block} with Proposition
\ref{One block} and noting that the partitions $\emptyset$ and $(1,1)$ are in the same block (since $W_2(\emptyset) \cong L_2(1,1)$). 
Observe that if $n\geq (p^2+7)/8$,
then $2r-1< p$ implies
\begin{align*}
 \frac{r(r+1)}{2}+p-2r \leq \frac{(p-1)(p+1)}{8} + 1 = \frac{p^2+7}{8}\leq n,
\end{align*}
since $r(r+1)/2 -2r$ is an increasing function of $r$ for $r\geq
2$. In particular there are no $r$ satisfying the condition of the
summation, and so we only have one block.
\end{proof}

\section*{Acknowledgement}
SB thanks Kevin Coulembier for helpful discussions and comments. 
SB is supported by a BOF Postdoctoral Fellowship from Ghent University.


\date{}

\end{document}